\documentclass[a4paper, notitlepage]{amsart}
\usepackage{amsmath, amssymb, amsthm, graphicx, xspace, mathrsfs, url}

\newcommand{\Z}{\ensuremath{\mathbb{Z}}\xspace}
\newcommand{\Q}{\ensuremath{\mathbb{Q}}\xspace}

\newcommand{\C}{\ensuremath{\mathbb{C}}\xspace}
\newcommand{\A}{\ensuremath{\mathbb{A}}\xspace}

\newcommand{\m}{\ensuremath{\mathfrak{m}}\xspace}
\newcommand{\M}{\ensuremath{\mathfrak{M}}\xspace}

\newcommand{\p}{\ensuremath{\mathfrak{p}}\xspace}

\newcommand{\OO}{\ensuremath{\mathscr{O}}\xspace}

\newcommand{\comment}[1]{}

\DeclareMathOperator{\Hom}{Hom}

\newtheorem{theorem}{Theorem}
\newtheorem{proposition}[theorem]{Proposition}

\newtheorem{lemma}[theorem]{Lemma}
\newtheorem{definition}[theorem]{Definition}

\newtheorem*{conjecture}{Conjecture} 
\newtheorem*{theor}{Theorem}

\newcommand{\OC}{\ensuremath{\mathbf{S}^D_{X}(U;r)}\xspace}

\newcommand{\OCdag}{\ensuremath{\mathbf{S}^D_{X}(U)^\dagger}\xspace}
\newcommand{\OCddag}{\ensuremath{\mathbf{V}^D_{X}(U)^\dagger}\xspace}

\newcommand{\OCN}{\ensuremath{\mathbf{S}^D_{X}(U_1(Np^\alpha);r)}\xspace}

\newcommand{\OCNl}{\ensuremath{\mathbf{S}^D_{X}(U_1(Np^\alpha)\cap U_0(l);r)}\xspace}

\newcommand{\OCV}{\ensuremath{\mathbf{S}^D_{X}(V;r)}\xspace}
\newcommand{\OCS}{\ensuremath{\mathbf{S}^D_{x}(U;r)}\xspace}

\newcommand{\OCdS}{\ensuremath{\mathbf{V}^D_{x}(U;r)}\xspace}
\newcommand{\OCVS}{\ensuremath{\mathbf{S}^D_{x}(V;r)}\xspace}
\newcommand{\OCVSd}{\ensuremath{\mathbf{V}^D_{x}(V;r)}\xspace}

\newcommand{\OCd}{\ensuremath{\mathbf{V}^D_{X}(U;r)}\xspace}

\newcommand{\OCdV}{\ensuremath{\mathbf{V}^D_{X}(V;r)}\xspace}
\newcommand{\s}{\ensuremath{^{\leq d}}\xspace}
\newcommand{\UU}{\ensuremath{U_1(Np^\alpha)}\xspace}
\newcommand{\VV}{\ensuremath{U_1(Np^\alpha)\cap U_0(l)}\xspace}
\input xy
\xyoption{all}

\begin{document}
\bibliographystyle{abbrv}
\title[Geometric level raising]{Geometric level raising for $p$-adic automorphic forms}
\author{James Newton}
\email{jjmn2@cam.ac.uk}
\begin{abstract}
We present a level raising result for families of $p$-adic automorphic forms for a definite quaternion algebra $D$ over $\Q$. The main theorem is an analogue of a theorem for classical automorphic forms due to Diamond and Taylor. We show that certain families of forms old at a prime $l$ intersect with families of $l$-new forms (at a non-classical point). One of the ingredients in the proof of Diamond and Taylor's theorem (which also played a role in earlier work of Taylor) is the definition of a suitable pairing on the space of automorphic forms. In our situation one cannot define such a pairing on the infinite dimensional space of $p$-adic automorphic forms, so instead we introduce a space defined with respect to a dual coefficient system and work with a pairing between the usual forms and the dual space. A key ingredient is an analogue of Ihara's lemma which shows an interesting asymmetry between the usual and the dual spaces.
\end{abstract}
\maketitle
\section{Introduction}
Classical level raising results typically show that if the reduction mod $p$ of a level $N$ modular form $f$ has certain properties (depending on a prime $l \ne p$), then there exists a modular form $g$ of level $Nl$, new at $l$, with $g \equiv f \mod p$. An example of a level raising result for classical modular forms is the following, due to Ribet \cite{MR804706}:
\begin{theor}
Let $f \in S_2(\Gamma_0(N))$ be an eigenform, and let
$\p|p$ be a finite place of $\bar{\Q}$ such that $p
\geq 5$ and $f$ is not congruent to an Eisenstein series modulo
$\p$. If $l \nmid Np$ is a prime number such that the following condition is satisfied,
$$
\text{$a_l(f)^2 \equiv (1+l)^2$ (mod $\p$),}
$$
then there exists a $l$-new eigenform $\tilde{f} \in
S_2(\Gamma_0(Nl))$ congruent to $f$ modulo $\p$.
\end{theor}
In this paper we prove an analogous level raising result for families of $p$-adic automorphic forms. In \cite{Bu1} and Part III of \cite{Bu2}, Buzzard defines modules of overconvergent $p$-adic automorphic forms for definite quaternion algebras, and constructs from these a so-called `eigencurve'. The eigencurve is a rigid analytic variety whose points correspond to certain systems of eigenvalues for Hecke algebras acting on these modules of automorphic forms. This space $p$-adically interpolates the systems of eigenvalues arising from classical automorphic forms. Emerton has constructed eigenvarieties in a cohomological framework \cite{Emint}, but in the following we will work with Buzzard's more concrete construction. We have also proved some cases of level raising for $p$-adic modular forms using the completed cohomology spaces investigated by Emerton (see \cite{cclr}).

The first construction of an eigencurve was carried out for modular forms (automorphic forms for $\mathrm{GL}_2$) in Coleman and Mazur's seminal paper \cite{CM}. An important recent result is the construction of a $p$-adic Jacquet-Langlands map between an eigencurve for a definite quaternion algebra and the $\mathrm{GL}_2$ eigencurve (interpolating the usual Jacquet-Langlands correspondence), as carried out in \cite{MR2111512}.

We follow the general approach of the first part of Diamond and Taylor's paper \cite{DT}, and our Theorem \ref{supporttheorem} is an analogue of \cite[Theorem 1]{DT}, but several new features appear in our work. In particular, the level raising results in \cite{Ta} and \cite{DT} for definite quaternion algebras are proved by utilising a pairing on finite dimensional vector spaces of automorphic forms. In our setting, the spaces of automorphic forms are Banach modules over an affinoid algebra, so we introduce spaces of `dual' automorphic forms and work with the pairing between the usual space of automorphic forms and the dual space. We then prove suitable forms of Ihara's lemma, our Theorem \ref{lemma:trivialtorsion} (cf. lemma 2 of \cite{DT}), for the usual and dual spaces of automorphic forms. An interesting asymmetry between the two situations can be observed.

This investigation of level raising results was motivated by a conjecture made by A. Paulin, prompted by results on local-global compatibility on the eigencurve in his thesis \cite{Pa}. Paulin's conjecture was made for the $\mathrm{GL}_2$-eigencurve; we may apply our theorem to the image of the $p$-adic Jacquet-Langlands map there to prove many cases of his conjecture. Since we have applications to the eigencurve for $\mathrm{GL}_2/\Q$ in mind we work with definite quaternion algebras over $\Q$ in this paper, but some of the methods of section \ref{mainbody} should apply to definite quaternion algebras over any totally real number field, although we do use the fact that weight space is one-dimensional in our arguments.
We end this introduction by stating the conjecture made by Paulin.

\subsection{A geometric level raising conjecture}
We fix two distinct primes $p$ and $l$, and an integer $N$ coprime to $pl$. Let $\mathscr{E}$ be the cuspidal Eigencurve of tame level $\Gamma_0(Nl)$, parameterising overconvergent cuspidal $p$-adic modular eigenforms (see \cite{Bu2} for its construction). If $\phi$ is a point of $\mathscr{E}$, corresponding to an eigenform $f_\phi$, Paulin defines an associated representation of $\mathrm{GL}_2(\Q_l)$, denoted $\pi_{f_\phi,l}$. We call an irreducible connected component $\mathscr{Z}$ of the Eigencurve \emph{generically special} if the $\mathrm{GL}_2(\Q_l)$-representations associated to the points of $\mathscr{Z}$ away from a discrete set are special. We define \emph{generically unramified principal series} similarly. Denote by $\alpha$ and $\beta$ the roots of the polynomial $X^2-t_lX+ls_l$, where $t_l$ and $s_l$ are the $T_l$ and $S_l$ eigenvalues of $f_\phi$. Paulin makes the conjecture:

\begin{conjecture}
Suppose $\mathscr{Z}$ is generically unramified principal series. Suppose further that there is a point $\phi$ on $\mathscr{Z}$ where the ratio of $\alpha$ and $\beta$ becomes $l^{\pm 1}$ and $\pi_{f_\phi,l}$ is special. Then there exists a generically special component $\mathscr{Z'}$ intersecting $\mathscr{Z}$ at $\phi$.
\end{conjecture}

Chenevier raised the same question (in a slightly different form) in relation to the characterisation of the Zariski closure of the $l$-new classical forms in the eigencurve. We address this issue in section \ref{neweigen}. Finally, in a recent preprint \cite{NonComp} Paulin has proved versions of his level raising (and lowering) conjectures (even for ramified principal series). His techniques are completely different to ours, making use of deformation theory and requiring a recent important result of Emerton \cite{Emlg} showing that the space $X_{fs}$ constructed by Kisin in \cite{Kis} is equal to the $\mathrm{GL}_2$-eigencurve (if one restricts to pieces of the two spaces where certain conditions are satisfied by the relevant mod $p$ Galois representations).
\section{Modules of $p$-adic overconvergent automorphic forms and Ihara's lemma}\label{mainbody}
In this section we will prove the results we need about modules of $p$-adic overconvergent automorphic forms for quaternion algebras.

\subsection{Banach modules}
Let $K$ be a finite extension of $\Q_p$. We call a normed $K$-algebra $A$ a \emph{Banach algebra} if it satisfies the following properties:
\begin{itemize}
\item $A$ is Noetherian,
\item the norm $|-|$ is non-Archimedean,
\item $A$ is complete with respect to $|-|$,
\item for any $x,$ $y$ in $A$ we have $|xy|\le |x||y|$.
\end{itemize}  We will normally assume $A$ is a reduced affinoid algebra with its supremum norm. A \emph{Banach $A$-module} is an $A$-module $M$ endowed with a norm $|-|$ such that
\begin{itemize}
\item for any $a \in A$, $m \in M$ we have $|am|\le |a||m|$,
\item $M$ is complete with respect to $|-|$.
\end{itemize}
Given a set $I$ we define the Banach $A$-module $c_I(A)$ to be functions $f: I \rightarrow A$ such that $lim_{i\rightarrow \infty}f(i)=0$, with norm the supremum norm. By a finite Banach $A$-module we mean a Banach $A$-module which is finitely-generated as an abstract $A$-module.

Suppose $M$ is a Banach module over a Banach algebra $A$. We say that $M$ is \emph{ONable} if it is isomorphic (as a Banach module) to some $c_I(A)$. Note that this terminology differs slightly from that of \cite{Bu2}, where \emph{ONable} refers to modules \emph{isometric} to some $c_I(A)$ and \emph{potentially ONable} replaces our notion of ONable. The Banach $A$-module $P$ is said to satisfy the universal property $(Pr)$ if for every surjection $f: M\rightarrow N$ of Banach $A$-modules and continuous map $\alpha:P \rightarrow N$, $\alpha$ lifts to a continuous map $\beta:P \rightarrow M$ such that the below diagram commutes:
$$
\xymatrix{ %Our diagram is a 2x2 matrix with the upper left entry emtpy
 & M\ar@{->>}[d]^f\\
P\ar@{.>}[ur]^{\exists \beta}\ar[r]^\alpha & N
}
$$

As explained in \cite[2.1.4]{Urban}, the universal property $(Pr)$ is the property of being projective in the category of Banach $A$-modules, where the notion of projective is defined using \emph{strict} epimorphisms of Banach modules (which are just the set-theoretically surjective epimorphisms of Banach modules). This is the correct notion of projective object, since the category of Banach $A$-modules is an exact category, not an Abelian catefory. A module $P$ having property $(Pr)$ is equivalent to $P$ being a direct summand of an ONable module. (See the end of section 2 in \cite{Bu2}).
\subsection{Some notation and definitions}\label{notdef}
Let $p$ be a fixed prime. Let $D$ be a definite quaternion algebra over $\Q$ with discriminant $\delta$ prime to $p$. Fix a maximal order $\OO_D$ of $D$ and isomorphisms $\OO_D \otimes \Z_q \cong M_2(\Z_q)$ for primes $q \nmid \delta$. Note that these induce isomorphisms $D \otimes \Q_q \cong M_2(\Q_q) $ for $q \nmid \delta$. We define $D_f=D\otimes_\Q\A_f$, where $\A_f$ denotes the finite adeles over $\Q$. Write $Nm$ for the reduced norm map from $D_f$ to $\A_f^\times$. Note that if $g \in D_f$ we can regard the $p$ component of $g$, $g_p$, as an element of $M_2(\Q_p)$. 

For an integer $\alpha \geq 1$, we let $\mathbb{M}_\alpha$ denote the monoid of matrices $\begin{pmatrix}
a & b \\ c & d \\ \end{pmatrix} \in M_2(\Z_p)$ such that $p^\alpha | c$,  $p \nmid d$ and $ad-bc \neq 0$. If $U$ is an open compact subgroup of $D_f^\times$ and $\alpha \geq 1$ we say that $U$ has \emph{wild level} $\geq p^\alpha$ if the projection of $U$ to $\mathrm{GL}_2(\Q_p)$ is contained in $\mathbb{M}_\alpha$.

We will be interested in two key examples of open compact subgroups of $D_f^\times$. For $M$ any integer prime to $\delta$, we define $U_0(M)$ (respectively $U_1(M)$) to be the subgroup of $D_f^\times$ given by the product $\prod_q U_q$, where $U_q = (\OO_D\otimes\Z_q)^\times$ for primes $q | \delta$, and $U_q$ is the matrices in $\mathrm{GL}_2(\Z_p)$ of the form $\begin{pmatrix}
\ast & \ast \\ 0 & \ast
\end{pmatrix}$ (respectively $\begin{pmatrix}
\ast & \ast \\ 0 & 1
\end{pmatrix}$) mod $q^{\mathrm{val}_q(M)}$ for all other $q$. We can see that if $p^\alpha$ divides $M$, then $U_1(M)$ has wild level $\geq p^\alpha$.

Suppose we have $\alpha \geq 1$, $U$ a compact open subgroup of $D_f^\times$ of wild level $\geq p^\alpha$ and $A$ a module over a commutative ring $R$, with an $R$-linear right action of $\mathbb{M}_\alpha$. We define an $R$-module $\mathscr{L}(U,A)$ by
$$\mathscr{L}(U,A)=\{f:D_f^\times \rightarrow A : f(dgu)=f(g)u_p\ \hbox{ for all }d \in D^\times, g\in D_f^\times, u \in U\}$$ where $D^\times$ is embedded diagonally in $D_f^\times$. If we fix a set $\{d_i : 1\leq i \leq r\}$ of double coset representatives for the finite double quotient $D^\times \backslash D_f^\times / U$, and write $\Gamma_i$ for the finite group $d_i^{-1}D^\times d_i \cap U$, we have an isomorphism (see section 4 of \cite{Bu1}) $$\mathscr{L}(U,A)\rightarrow \bigoplus_{i=1}^r A^{\Gamma_i},$$ given by sending $f$ to $(f(d_1),f(d_2),\ldots,f(d_r))$. If $U \subset U_1(N)$ for $N \ge 4$, then the groups $\Gamma_i$ are trivial (this is proved in \cite{DT}).

For $f : D_f^\times \rightarrow A$, $x\in D_f^\times$ with $x_p \in \mathbb{M}_\alpha$, we define $f|x : D_f^\times \rightarrow A$ by $(f|x)(g)=f(gx^{-1})x_p$. Note that we can now also write $$\mathscr{L}(U,A)=\{f:D^\times \backslash D_f^\times \rightarrow A : f|u=f \hbox{ for all }u \in U\}.$$

We can define double coset operators on the spaces $\mathscr{L}(U,A)$. If $U$, $V$ are two compact open subgroups of $D_f^\times$ of wild level $\geq p^\alpha$, and $A$ is as above, then for $\eta\in D_f^\times$ with $\eta_p \in \mathbb{M}_\alpha$ we may define an $R$-module map $[U\eta V]:\mathscr{L}(U,A) \rightarrow \mathscr{L}(V,A)$ as follows: we decompose $U\eta V$ into a finite union of right cosets $\coprod_i U x_i$ and define $$f|[U\eta V] = \sum_i f|x_i.$$

\subsection{Overconvergent automorphic forms}\label{ocforms}

Let $\mathscr{W}$ be the rigid analytic space $\Hom(\Z_p^\times, \mathbb{G}_m)$, defined over $\Q_p$. The reader may consult lemma 2 of \cite{Bu1} for details of this space's construction and properties. For example, $\mathscr{W}$ is a union of finitely many open discs. The space $\mathscr{W}$ is the \emph{weight space} for our automorphic forms. The $\C_p$-points $w$ of $\mathscr{W}$ corresponding to characters $\kappa_w: \Z_p^\times \rightarrow \C_p^\times$ with $\kappa_w(x)=x^k \varepsilon_p(x)$ for some positive integer $k$ and  finite order character $\varepsilon_p$ are referred to as \emph{classical} weights. Let $X$ be a reduced connected $K$-affinoid subspace of $\mathscr{W}$, where $K/\Q_p$ is finite, and denote the ring of analytic functions on $X$ by $\OO(X)$. Such a space $X$ corresponds to a character $\kappa: \Z_p^\times \rightarrow \OO(X)^\times$ induced by the inclusion $X \subset \mathscr{W}$. If we have a real number $r=p^{-n}$ for some $n$, then we define $\mathbb{B}_{r,K}$ to be the rigid analytic subspace of affine $1$-space over $K$ with $\C_p$-points $$\mathbb{B}_{r,K}(\C_p)=\{z \in \C_p: \exists y\in \Z_p \hbox{ such that } |z-y|\leq r\}.$$ Similarly (for $r < 1$) we define $\mathbb{B}_{r,K}^\times$ to be the rigid analytic subspace of affine $1$-space over $K$ with $\C_p$-points $$\mathbb{B}_{r,K}^\times(\C_p)=\{z \in \C_p: \exists y\in \Z_p^\times \hbox{ such that } |z-y|\leq r\}.$$
A point $x \in X(\C_p)$ corresponds to a continuous character $\kappa_x :\Z_p^\times \rightarrow \C_p^\times$. Such maps are analytic when restricted to the set $\{z \in \Z_p : |1-z| \le r \}$ for small enough $r$. If $\kappa_x$ and $r$ have this property we call $x$ an $r$-analytic point. A point is $r$-analytic if and only if its corresponding character extends to a morphism of rigid analytic varieties $$\kappa_x:\mathbb{B}_{r,K}^\times \rightarrow \mathbb{G}_m.$$ 

Let $X$ be a $K$-affinoid subspace of $\mathscr{W}$ as before, with associated character $\kappa: \Z_p^\times \rightarrow \OO(X)^\times$. We say that $\kappa$ is $r$-analytic if every point in $X(\C_p)$ is $r$-analytic. Fix a real number $0 < r < 1$ and let $\mathscr{A}_{X,r}$ be the $\OO(X)$-Banach algebra $\OO(\mathbb{B}_{r,K} \times_K X)$, endowed with the supremum norm. If $\kappa$ is $rp^{-\alpha}$-analytic we can define a right action of $\mathbb{M}_\alpha$ on $\mathscr{A}_{X,r}$ by, for $f \in \mathscr{A}_{X,r}$, $\gamma = \begin{pmatrix}
a&b\\
c&d
\end{pmatrix} \in \mathbb{M}_\alpha ,$
$$(f\cdot \gamma)(x,z)=\frac{\kappa_x(cz+d)}{(cz+d)^2}f\left (x,\frac{az+b}{cz+d}\right ).$$ where $x \in X(\C_p)$ (with $\kappa_x$ the associated character) and $z \in \mathbb{B}_{r,K}(\C_p)$.

\begin{definition} \label{ocformdef}Let $X$ be a $K$-affinoid subspace of $\mathscr{W}$ as above, with $\kappa:\Z_p^\times \rightarrow \OO(X)^\times$ the induced character. If we have a real number $r=p^{-n}$, some integer $\alpha \ge 1$ such that $\kappa$ is $rp^{-\alpha}$-analytic, and $U$ a compact open subgroup of $D_f^\times$ of wild level $\geq p^{\alpha}$, then define the space of $r$-overconvergent automorphic forms of weight $X$ and level $U$ to be the $\OO(X)$-module
$$\mathbf{S}^D_{X}(U;r):=\mathscr{L}(U,\mathscr{A}_{X,r}).$$ 
\end{definition}

If we endow $\mathbf{S}^D_{X}(U;r)$ with the norm $|f|=\max_{g\in D_f^\times}|f(g)|$, then the isomorphism \begin{equation}\mathbf{S}^D_{X}(U;r)\cong \bigoplus_{i=1}^r \mathscr{A}_{X,r}^{\Gamma_i}\label{iso}\end{equation} induced by fixing double coset representatives $d_i$ is norm preserving. Since the $\Gamma_i$ are finite groups, and $\mathscr{A}_{X,r}$ is an ONable Banach $\OO(X)$-module (it is the base change to $\OO(X)$ of $\OO(\mathbb{B}_{r,K})$, and all Banach spaces over a discretely valued field are ONable), we see that $\mathbf{S}^D_{X}(U;r)$ is a Banach $\OO(X)$-module, and satisfies property $(Pr)$.

Note that if $\m$ is a maximal ideal of $\OO(X)$, corresponding to a point $x \in X(K')$ for $K'/K$ finite, then taking the fibre of the module $\mathbf{S}^D_{X}(U;r)$ at $\m$ gives the space of overconvergent forms $\mathbf{S}^D_{x}(U;r)$ corresponding to the point $x$ of $\mathscr{W}(K')$ (note that a point of $\mathscr{W}(K')$ is a reduced connected $K'$-affinoid subspace!).

These spaces of overconvergent automorphic forms were first defined in \cite{Bu1}, using ideas from the unpublished preprint \cite{St}.

\subsection{Dual modules}
Suppose $A$ is a Banach algebra. Given a Banach $A$-module $\mathbf{M}$ we define the \emph{dual} $\mathbf{M}^*$ to be the Banach $A$-module of continuous $A$-module morphisms from $\mathbf{M}$ to $A$, with the usual operator norm. We denote the $\OO(X)$-module $\mathscr{A}_{X,r}^*$ by $\mathscr{D}_{X,r}$.

If the map $\kappa$ corresponding to $X$ is $rp^{-\alpha}$-analytic, then $\mathbb{M}_\alpha$ acts continuously on $\mathscr{A}_{X,r}$, so $\mathscr{D}_{X,r}$ has an $\OO(X)$-linear right action of the monoid $\mathbb{M}_\alpha^{-1}$ given by $(f\cdot m^{-1})(x):=f(x\cdot m)$, for $f\in \mathscr{D}_{X,r}$, $x \in \mathscr{A}_{X,r}$ and $m \in \mathbb{M}_\alpha$. If $U$ is as in Definition~\ref{ocformdef} then its projection to $\mathrm{GL}_2(\mathbb{Q}_p)$ is contained in $\mathbb{M}_\alpha \cap \mathbb{M}_\alpha^{-1}$, so it acts on  $\mathscr{D}_{X,r}$. This allows us to make the following definition:

\begin{definition}
For $X$, $\kappa$, $r$, $\alpha$ and $U$ as above, we define the space of \emph{dual} \\$r$-overconvergent automorphic forms of weight $X$ and level $U$ to be the $\OO(X)$-module
$$\mathbf{V}^D_{X}(U;r):=\mathscr{L}(U,\mathscr{D}_{X,r}).$$
\end{definition}

As in \ref{ocforms}, we have a norm preserving isomorphism \begin{equation}\mathbf{V}^D_{X}(U;r)\cong \bigoplus_{i=1}^r \mathscr{D}_{X,r}^{\Gamma_i}.\label{dualiso}\end{equation} Thus $\mathbf{V}^D_{X}(U;r)$ is a Banach $\OO(X)$-module. We note that it will not usually satisfy property $(Pr)$, since (unless $X$ is a point)  we expect that $\mathscr{D}_{X,r}$ will not be ONable.

If $U$, $V$ are two compact open subgroups of $D_f^\times$ of wild level $\geq p^\alpha$, then for $\eta\in D_f^\times$ with $\eta_p \in \mathbb{M}_\alpha^{-1}$ we get double coset operators $[U\eta V]:\mathbf{V}^D_{X}(U;r) \rightarrow \mathbf{V}^D_{X}(V;r)$.

\subsection{Hecke operators}
For an integer $m$, we define the \emph{Hecke algebra away from} $m$, $\mathbb{T}^{(m)}$,  to be the free commutative $\OO(X)$-algebra generated by symbols $T_\pi, S_\pi$ for $\pi$ prime not dividing $m$. If $\delta p$ divides $m$ then we can define the usual action of $\mathbb{T}^{(m)}$ by double coset operators on $\OC$: for $\pi\nmid m$ define $\varpi_\pi\in\A_f$ to be the finite adele which is $\pi$ at $\pi$ and $1$ at the other places. Abusing notation slightly, we also write $\varpi_\pi$ for the element of $D_f^\times$ which is $\begin{pmatrix}
\pi & 0\\0&\pi
\end{pmatrix}$ at $\pi$ and the identity elsewhere. Similarly set $\eta_\pi=\begin{pmatrix}
\varpi_\pi & 0\\0&1
\end{pmatrix}$ to be the element of $D_f^\times$ which is $\begin{pmatrix}
\pi & 0\\0&1
\end{pmatrix}$ at $\pi$ and the identity elsewhere. On $\OC$ we let $T_\pi$ act by $[U\eta_\pi U]$ and $S_\pi$ by $[U\varpi_\pi U]$.
Similarly on $\OCd$ we define $T_\pi$ to act by $[U\eta_\pi^{-1}U]$ and $S_\pi$ by $[U\varpi_\pi^{-1} U]$. 
As usual we also have a compact operator acting on $\OC$, namely $U_p:=[U \eta_p U]$.
\subsection{A pairing}
In this section $X$, $\kappa$, $r$, $\alpha$ and $U$ will be as in Definition~\ref{ocformdef}. We will denote by $V$ another compact open subgroup of wild level $\geq p^\alpha$. We fix double coset representatives $\{d_i : 1\leq i \leq r\}$ for the double quotient $D^\times \backslash D_f^\times / U$ and let $\gamma_i$ denote the order of the finite group $d_i^{-1}D^\times d_i \cap U$. We can define an $\OO(X)$-bilinear pairing between the spaces $\OC$ and $\OCd$ by
$$\langle f,\lambda \rangle:=\sum_{i=1}^r \gamma_i^{-1}\langle f(d_i), \lambda(d_i)\rangle,$$
where $f \in \OC$, $\lambda \in \OCd$ and on the right hand side of the above definition $\langle\cdot , \cdot \rangle$ denotes the pairing between $\mathscr{A}_{X,r}$ and $\mathscr{D}_{X,r}$ given by evaluation.

This pairing is independent of the choice of the double coset representatives $d_i$, since for every $d \in D^\times$, $g \in D^\times_f$, $u \in U$, $f\in \OC$ and $\lambda \in \OCd$ we have $$\langle f(dgu),\lambda(dgu)\rangle=\langle f(g)u_p,\lambda(g)u_p \rangle=\langle f(g)u_p u_p^{-1},\lambda(g)\rangle=\langle f(g),\lambda(g)\rangle.$$ Combining this observation with the isomorphisms (\ref{iso}) and (\ref{dualiso}) we see that our pairing identifies $\OCd$ with $\OC^*$.

The following proposition summarises a standard computation \cite{Ta,DT} (although these assume the level group is small enough that the finite groups $\Gamma_i$ are trivial), telling us how our pairing interacts with double coset operators. In particular, it implies that $\langle T_\pi f, \lambda \rangle = \langle f, T_\pi \lambda \rangle$ for $\pi \nmid \delta p$ when $T_\pi$ acts in the usual way.
\begin{proposition}\label{prop:pairinghecke}
Let $f \in \OC$ and let $\lambda\in \mathbf{V}^D_X(V;r)$. Let $g \in D^\times_f$ with $g_p\in\mathbb{M}_\alpha$. Then $$\langle f|[U g V],\lambda\rangle=\langle f,\lambda|[Vg^{-1}U]\rangle.$$
\end{proposition}
\begin{proof}
For $d \in D^\times_f$ set $\gamma(d)=\#(d^{-1}D^\times d \cap V)$. We have $$f|[Ug V]=\sum_{v\in (g^{-1} U\eta)\cap V \backslash V}f|(g v),$$ hence \begin{eqnarray*}\langle f|[Ug V],\lambda\rangle&=&\sum_{d\in D^\times\backslash D^\times_f/V} \gamma(d)^{-1}\langle f|[Ug V](d),\lambda(d)\rangle\\
&=&\sum_{d\in D^\times\backslash D^\times_f/V}\sum_{v\in (g^{-1} Ug)\cap V \backslash V} \gamma(d)^{-1}\langle f|(g v)(d),\lambda(d)\rangle\\
&=&\sum_{d\in D^\times\backslash D^\times_f/V}\sum_{v\in (g^{-1} Ug)\cap V \backslash V} \gamma(d)^{-1}\langle f(dv^{-1}g^{-1})\cdot g_p v_p,\lambda(d)\rangle\\
&=&\sum_{x\in D^\times\backslash D^\times_f/(g^{-1} Ug)\cap V}\langle f(xg^{-1}),\lambda(x)\cdot g_p^{-1}\rangle\\
&=&\sum_{y\in D^\times\backslash D^\times_f/U\cap(g V g^{-1})}\langle f(y),\lambda(yg)\cdot g_p^{-1}\rangle\\
&=&\langle f,\lambda|[Vg^{-1}U]\rangle
\end{eqnarray*}where we pass from the third line to the fourth line by counting  double cosets and the final line follows by similar calculations to the first $5$ lines.
\end{proof}

By the results of section 5 in \cite{Chenun} and section 3 of \cite{Bu2} we know that for a fixed $d\geq 0$, if $X$ is a sufficiently small affinoid whose norm is multiplicative (with a precise bound given by Th\'eor\`eme 5.3.1 of \cite{Chenun}) then since $U_p$ acts as a compact operator on $\OC$, we have a $U_p$ stable decomposition $$\OC=\OC^{\leq d}\oplus \mathbf{N},$$ where $\OC^{\leq d}$ is the space of forms of slope $\leq d$. We need $X$ to be small enough that the Newton polygon of the characteristic power series for $U_p$ acting on $\OC$ has the same slope $\le d$ part when specialised to any point of $X$.

From now on we fix $d$ and assume that $X$ is such that this slope decomposition exists. The key example of such an $X$ is an open ball of small radius.

The space $\OC^{\leq d}$ is a finite Banach $\OO(X)$-module with property $(Pr)$, i.e. a projective finitely generated $\OO(X)$-module. In fact this decomposition must be stable under the action of $\mathbb{T}^{(\delta p)}$, since the $T_\pi$ and $S_\pi$ operators for $\pi \neq p$ commute with $U_p$. We define $\OCd^{\leq d}$ to be the maps from $\OC$ to $\OO(X)$ which are $0$ on $\mathbf{N}$. This space is also stable under the action of $\mathbb{T}^{(\delta p)}$ and is naturally isomorphic to the dual of $\OC^{\leq d}$. The following lemma implies that our pairing is perfect when restricted to $\OC^{\leq d}\times \OCd^{\leq d}$. 

\begin{lemma}
Let $\mathbf{M}$ be a finite Banach $\OO(X)$-module with property $(Pr)$. Then the usual natural map $\mathbf{M}\rightarrow (\mathbf{M}^{*})^{*}$ is an isomorphism. In other words, the $\OO(X)$-module $\mathbf{M}$ is \emph{reflexive}.
\end{lemma}
\begin{proof}
Since $\mathbf{M}$ is finite we have a surjection of Banach $\OO(X)$-modules $\OO(X)^{\oplus n}\rightarrow \mathbf{M}$ for some $n$. Applying the universal property $(Pr)$ to this surjection shows that we have a Banach $\OO(X)$-isomorphism $\mathbf{M}\oplus \mathbf{N} \cong \OO(X)^{\oplus n}$ for some module $\mathbf{N}$, so $\mathbf{M}$ is a projective $\OO(X)$-module. Proposition 2.1 of \cite{Bu2} states that the category of finite Banach $\OO(X)$-modules, with continuous
$\OO(X)$-linear maps as morphisms, is equivalent to the category of
finite $\OO(X)$-modules so we can just compute duals module-theoretically. We have exact sequences
$$\xymatrix{0\ar[r]&\mathbf{M}\ar[r]&\OO(X)^{\oplus n}\ar[r]&\mathbf{N}\ar[r]&0} $$
$$\xymatrix{0\ar[r]&\mathbf{N}\ar[r]&\OO(X)^{\oplus n}\ar[r]&\mathbf{M}\ar[r]&0} $$

and since $\mathbf{M}$, $\mathbf{N}$, $\mathbf{M}^*$ and $\mathbf{N}^*$ are all projective as $\OO(X)$-modules we can take the dual of these exact sequences twice to get commutative diagrams with exact rows
$$\xymatrix{
0\ar[r]&\mathbf{M}\ar[r]\ar[d]&\OO(X)^{\oplus n}\ar[r]\ar[d]&\mathbf{N}\ar[r]\ar[d]&0\\
0\ar[r]&\mathbf{M}^{**}\ar[r]&\OO(X)^{\oplus n}\ar[r]&\mathbf{N}^{**}\ar[r]&0\\
}$$
$$\xymatrix{
0\ar[r]&\mathbf{N}\ar[r]\ar[d]&\OO(X)^{\oplus n}\ar[r]\ar[d]&\mathbf{M}\ar[r]\ar[d]&0\\
0\ar[r]&\mathbf{N}^{**}\ar[r]&\OO(X)^{\oplus n}\ar[r]&\mathbf{M}^{**}\ar[r]&0\\
}$$
where the vertical maps are the natural maps from a module to its double dual. Since the central maps are isomorphisms, we conclude that the outer maps are too.
\end{proof} 

\subsubsection{Direct limits and Fr\'echet spaces}
We should remark here that our use of the dual Banach modules $\OCd$ is slightly unsatisfactory. For example, the modules do not satisfy property $(Pr)$, and we must restrict to `slope $\leq d$' subspaces to get a perfect pairing. One could alternatively work with modules of \emph{all} overconvergent automorphic forms, rather than imposing $r$-overconvergence for a particular $r$. One defines
$$\OCdag :=  \varinjlim_r \OC,$$
where the (compact) transition maps in the direct system are induced by the inclusions $\mathbb{B}_{s,K} \subset \mathbb{B}_{r,K}$ for $s < r$. 
If $X$ is a point (so $\OO(X)$ is a field) then it is a standard result that the vector space $\OCdag$ is reflexive (see Proposition 16.10 of \cite{NFA}). Using this fact it is fairly straightforward to show that for any $X$, $\OCdag$ is a reflexive $\OO(X)$-module, with dual the Fr\'echet space
$$\OCddag := \varprojlim_r \OCd.$$
\subsection{Old and new}\label{oldnew}
Fix an integer $N\ge 1$ (the tame level) coprime to $p$ and fix an auxiliary prime $l \nmid Np\delta$. Let $X$ be an affinoid subspace of weight space with associated character $\kappa$ which is $rp^{-\alpha}$ analytic, for some integer $\alpha \ge 1$. Set $U=\UU$, $V=\VV$. To simplify notation we set \begin{eqnarray*}L&:=&\OC^{\leq d}\\L^*&:=&\OCd\s \\M&:=&\OCV\s\\M^*&:=&\OCdV\s.\end{eqnarray*} 

We define a map $i:L \times L \rightarrow M$ by
$$i(f,g):=f|[U 1 V]+g|[U \eta_l V].$$ Since the map $i$ is defined by double coset operators with trivial component at $p$ it commutes with $U_p$ and thus gives a well defined map between these spaces of bounded slope forms.
A simple calculation shows that these double coset operators act very simply. Regarding $f$ and $g$ as functions on $D^\times_f$ we have $f|[U 1 V]=f$, $g|[U\eta_l V]=g|\eta_l$. The image of $i$ inside $M$ will be referred to as the space of \emph{oldforms}.

We also define a map $i^\dagger:M \rightarrow L \times L$ by
$$i^\dagger (f):=(f|[V 1 U],f|[V \eta_l^{-1} U]).$$ The kernel of $i^\dagger$ is the space of \emph{newforms}. The maps $i$ and $i^\dagger$ commute with Hecke operators $T_q, S_q$, where $q \nmid Npl\delta$.

The same double coset operators give maps $$j:L^* \times L^* \rightarrow M^*,$$ $$j^\dagger: M^* \rightarrow L^* \times L^*.$$

Using Proposition \ref{prop:pairinghecke} we have 
$$\langle i(f,g),\lambda\rangle = \langle (f,g), j^\dagger \lambda \rangle$$
for $f,g\in L$, $\lambda \in M^*$. Similarly
$$\langle f,j(\lambda,\mu)\rangle = \langle i^\dagger f, (\lambda,\mu) \rangle$$ for $d \in M$, $\lambda,\mu \in L^*$.

An easy calculation shows that $i^\dagger i$ acts on the product $L\times L=L^2$ by the matrix (acting on the right) $$\begin{pmatrix}l+1 & [U\varpi_l^{-1}U][U\eta_l U] \\ [U\eta_l U] & l+1 \end{pmatrix}=\begin{pmatrix}l+1 & S_l^{-1}T_l \\ T_l & l+1 \end{pmatrix}.$$ We have exactly the same double coset operator formula for the action of $j^\dagger j$ on the product $L^*\times L^* = L^{*2}$. The Hecke operators $S_l, T_l$ act by $[U\varpi_l^{-1}U], [U\eta_l^{-1} U]$ respectively on $L^*$. Also, the double coset $U\eta_l U$ is the same as $U\varpi_l \eta_l^{-1}U$, since the matrix which is the identity at every factor except $l$ and $\begin{pmatrix}
0 & 1\\1&0
\end{pmatrix}$ at $l$ is in U. From these two facts we deduce that, in terms of Hecke operators, $j^\dagger j$ acts on $L^*\times L^*$ by the matrix (again acting on the right) $$\begin{pmatrix}l+1 & T_l \\ S_l^{-1}T_l & l+1 \end{pmatrix}.$$

If the affinoid $X$ is sufficiently nice, then we can show that the map $i^\dagger i$ is injective. Before we prove this, we note that in our setting a family of $p$-adic automorphic eigenforms over an affinoid $X \subset \mathscr{W}$ is just a Hecke eigenform $f$ in $\OC$.
\begin{proposition}\label{propinj} If $X$ is a one dimensional irreducible connected smooth affinoid, then the map $i^\dagger i$ is injective. \end{proposition}
\begin{proof}
Let $L_0$ be the projective (since $\OO(X)$ is a Dedekind domain) finite Banach $\OO(X)$-module $\ker(i^\dagger i)$, and note that $L_0\subset L^2$ is stable under the action of all the Hecke operators, since they all commute with $i^\dagger i$. Suppose $L_0$ is not zero. For $(f,g)$ in $L_0$ we have $(l+1)f+S_l^{-1}T_l g=T_l f+(l+1)g=0$. Eliminating $g$ we get $T_l^2 f-(l+1)^2 S_l f=0$, so projecting $L_0$ down to $L$ (taking either the first or the second factor) we see that the Hecke operator $T_l^2-(l+1)^2 S_l$ acts as $0$ on a non-zero projective submodule of $L$. This (applying the local eigenvariety construction as described in section 6.2 of \cite{Chenun}) implies that there is a family of eigenforms over some one dimensional sub-affinoid of $X$, all with the eigenvalue of $T_l^2-(l+1)^2 S_l$ equal to $0$. Now the Hecke algebra element $T_l^2-(l+1)^2 S_l$ induces a rigid analytic function on the tame level $N$ eigencurve for $D$ (by taking the appropriate eigenvalue associated to a point), so this function must vanish on the whole irreducible component containing the one dimensional family constructed above.  However, every irreducible component contains a classical point, and these cannot be contained in the kernel of $T_l^2-(l+1)^2 S_l$ since this would contradict the Hecke eigenvalue bounds given by the Ramanujan-Petersson conjecture. \end{proof}

Note that the injectivity of $i^\dagger i$ implies the injectivity of $i$. The above shows that if $X$ is as in the statement of Proposition \ref{propinj}, we have $\ker(i^\dagger)\cap \mathrm{im}(i)=0$ so our families in $M$ are not both old and new at $l$. However, if $X$ is just a point, then $i^\dagger i$ may have a kernel - this corresponds to $p$-adic automorphic forms which are both old and new at $l$.

\subsection{Some modules}\label{somemodules}
We denote the fraction field of $\OO(X)$ by $F$. If $A$ is an $\OO(X)$-module we write $A_F$ for the $F$-vector space $A\otimes_{\OO(X)}F$.

We begin this section by noting that the injectivity of $i^\dagger i$ implies the injectivity of $j^\dagger j$:

Suppose $j^\dagger j (\lambda,\mu)=0$. Then $\langle (f,g),j^\dagger j (\lambda,\mu)\rangle=0$ for all $(f,g)\in L_F$, so (by Proposition \ref{prop:pairinghecke}) $\langle i^\dagger i(f,g),(\lambda,\mu)\rangle =0$ for all $(f,g)\in L_F$. Now since $i^\dagger i:L_F\rightarrow L_F$ is an injective endomorphism of a finite dimensional vector space, it is an isomorphism, so we see that $\lambda=\mu=0$. Hence $j^\dagger j$ (thus a fortiori $j$) is injective.

We now define two chains of modules which will prove useful:
$$\begin{array}{l l}\Lambda_0:=L^2 & \Lambda_0^*:=L^{*2}\\
\Lambda_1:=i^\dagger M & \Lambda_1^*:=j^\dagger M^*\\
\Lambda_2:=i^\dagger(M\cap i(L_F^2)) & \Lambda_2^*:=j^\dagger(M^*\cap j((L^*_F)^2))\\
\Lambda_3:=i^\dagger i L^2 & \Lambda_3^*:=j^\dagger j L^{*2}.\end{array}$$
We note that $\Lambda_0\supset \Lambda_1 \supset \Lambda_2 \supset \Lambda_3$, and that $$\Lambda_2/\Lambda_3=i^\dagger(M\cap i(L_F^2)/iL^2)=i^\dagger((M/iL^2)^{\mathrm{tors}}),$$ with analogous statements for the starred modules.

We fix the usual action of $\mathbb{T}^{(N\delta p l)}$ on all these modules. We can now describe some pairings between them which will be equivariant under the $\mathbb{T}^{(N\delta p l)}$ action. They will not all be equivariant with respect to the action of $T_l$.

We have a (perfect) pairing $\langle,\rangle:L_F^2\times (L^*_F)^2\rightarrow F$ which, since $j$ is injective, induces a pairing $$\Lambda_0 \times (M^*\cap j((L^*_F)^2)) \rightarrow F/\OO(X),$$ which in turn induces a pairing $$P_1: \Lambda_0/\Lambda_1 \times (M^*\cap j((L^*_F)^2)/j(L^{*2})) \rightarrow F/\OO(X).$$ The fact that this pairing is perfect follows from the following lemma:

\begin{lemma}
The pairing on $L_F^2\times (L^*_F)^2$ induces isomorphisms $$\Hom_{\OO(X)}(\Lambda_1,\OO(X))\cong M^*\cap j((L^*_F)^2)$$ and $$\Hom_{\OO(X)}(\Lambda_0,\OO(X))\cong j(L^{*2}).$$\end{lemma}
\begin{proof}
For the first isomorphism, elements of the module $\Hom_{\OO(X)}(\Lambda_1,\OO(X))$ correspond to $l \in (L^*_F)^2$ such that $\langle i^\dagger m,l\rangle \in \OO(X)$ for all $m \in M$. We have $\langle i^\dagger m,l\rangle =\langle m, jl \rangle$ so $\langle i^\dagger m,l\rangle \in \OO(X)$ for all $m \in M$ if and only if $\langle m, jl \rangle \in \OO(X)$ for all $m \in M$, i.e. if and only if $jl \in M^*$. 

The second isomorphism is obvious, since $j$ is injective.
\end{proof}
In exactly the same way, we have a perfect pairing 
$$P_2: (M\cap i(L_F^2))/i(L^2) \times \Lambda_0^*/\Lambda_1^* \rightarrow F/\OO(X).$$

The final pairing we will need is induced by the pairing between $M$ and $M^*$. It is straightforward to check that this gives a perfect pairing: $$P_3: \mathrm{ker}(i^\dagger) \times M^*/(M^* \cap j((L_F^*)^2)) \rightarrow \OO(X).$$ 
\subsection{An analogue of Ihara's lemma}\label{Ihara}

In classical level raising results (such as \cite{DT,MR804706,Ta}) analogues of `Ihara's lemma' (Lemma 3.2 in \cite{Ih}) are used to show that prime ideals of a Hecke algebra containing the annihilators of certain modules of automorphic forms are in some sense `uninteresting', or even to show that these modules are trivial. In this section we prove the appropriate analogue of Ihara's lemma in our setting.

From this section onwards we will assume that $X$ is a one dimensional irreducible connected smooth affinoid in weight space $\mathscr{W}$, so we can apply Proposition \ref{propinj}. We want to obtain information about the $\mathbb{T}^{(N\delta p l)}$ action on the quotients $\Lambda_2/\Lambda_3\cong i^\dagger(M/iL^2)^{\mathrm{tors}}$, $\Lambda_2^*/\Lambda_3^*\cong j^\dagger(M^*/jL^{*2})^{\mathrm{tors}}$, $\Lambda_0/\Lambda_1$ and $\Lambda_0^*/\Lambda_1^*$. The pairings $P_1$ and $P_2$ allow us to use an analogue of Ihara's lemma (the following two propositions and theorem) to obtain crucial information about all four quotients. Recall that the radius of overconvergence $r$ equals $p^{-n}$ for some positive integer $n$. Fix a positive integer $c$ such that $Nm(U_1(Np^{\alpha+n}))$ contains all elements of $\hat \Z^{\times}$ congruent to $1$ modulo $c$. We first need a lemma allowing us to control certain forms with weight a point in weight space.

\begin{lemma}\label{normtrivial} Let $x \in \mathscr{W}(K')$ for some $K'$ a finite extension of $\Q_p$.
\begin{enumerate} \item Let $y\in \OCS$ be non-zero. Suppose $y$ factors through $Nm$, that is $y(g)=y(h)$ for all $g, h \in D_f^\times$ with $Nm(g)=Nm(h)$. Then $\kappa_x$ is a classical weight $z \mapsto z^2 \varepsilon_p(z)$, and for all but finitely many primes $q\equiv 1$ mod $c$, (where $c$ is the fixed integer chosen above), $(T_q-q-1)y=0$.
\item Let $y\in \OCdS$. If $y$ factors through $Nm$, that is $y(g)=y(h)$ for all $g, h \in D_f^\times$ with $Nm(g)=Nm(h)$, then $y$ is zero.
\end{enumerate}\end{lemma}

\begin{proof}
We first prove part (i). Suppose $y$ is as in the statement of that part. For $u_p \in \mathrm{SL}_2(\Q_p)\cap U$ we have $y(g)=y(gu_p)=y(g)\cdot u_p$ for all $g \in D^\times_f$. Noting that $\begin{pmatrix}
1 & a\\ 0 & 1
\end{pmatrix}\in \mathrm{SL}_2(\Q_p)\cap U$ for all $a \in \Z_p$, we see that $y(g)(z+a)=y(g)(z)$ for all $a \in \Z_p$, $z \in \mathbb{B}_{r,K'}$ so $y(g)(z)$ is constant in $z$, since non-constant rigid analytic functions have discrete zero sets. Recall that $U=\UU$, so $u_0:=\begin{pmatrix}
1 & 0\\ p^\alpha & 1
\end{pmatrix}$ is in $\mathrm{SL}_2(\Q_p)\cap U$, and for $z \in \mathbb{B}_{r,K'}$ we have $$y(g)(z)=(y(g)u_0)(z) = \frac{\kappa_x(p^\alpha z + 1)}{(p^\alpha z + 1)^2}y(g),$$ so $\kappa_x$ must correspond to the classical weight given by $z \mapsto z^2\epsilon_p(z)$ for some character $\epsilon_p$ trivial on $1+p^{\alpha + n} \Z_p$, where $r=p^{-n}$. This now implies that for each $g \in D^\times_f$ we have $y(g)\gamma = y(g)$ for all $\gamma$ in the projection of $U_1(Np^{\alpha+n})$ to $\mathbb{M}_{\alpha + n}$, since these matrices all have bottom right hand entry congruent to $1 \mod p^{\alpha+n}$. 

We now follow \cite{DT} to complete the proof of the first part of the lemma. There is a $d_0 \in D^{\times}$ with $Nm(d_0)=q$, so $Nm(d_0^{-1}\eta_q)\in \A_f^\times$ is actually in $\hat \Z^{\times}$ and is congruent to $1$ mod $c$. Thus (by the way we picked $c$) there is $u_0 \in U_1(Np^{\alpha+n})$ such that $Nm(u_0)=Nm(d_0^{-1}\eta_q)$. Now we have

\begin{align*}T_q(y)(g)&=\sum_{u\in(\eta_q^{-1} U\eta_q)\cap U\backslash U}y(gu^{-1} \eta_q^{-1})\cdot u_p\\
&=\sum_{u\in(\eta_q^{-1} U\eta_q)\cap U\backslash U}y(g\eta_q^{-1}u^{-1})\cdot u_p\\
&=\sum_{u\in(\eta_q^{-1} U\eta_q)\cap U\backslash U}y(g\eta_q^{-1})=(q+1)y(g\eta_q^{-1})\\
&=(q+1)y(g\eta_q^{-1}d_0 d_0^{-1})=(q+1)y(g u_0^{-1}d_0^{-1})=(q+1)y(g u_0^{-1})=(q+1)y(g),\end{align*}
where to pass from the first line to the second we use the fact that $y$ factors through $Nm$ to commute $y$'s arguments, from the second to the third we use that $y$ is modular of level $U$ and in the final line we first substitute $u_0$ for $d_0^{-1}\eta_q$ (since they have the same reduced norm), then commute $y$'s arguments and use the left invariance of $y$ under $D^\times$ followed by the fact that $u_0^{-1} \in U_1(Np^{\alpha+n})$ implies that $y(gu_0^{-1})=y(g)u_{0,p}^{-1}=y(g)$.

We now give a proof of the second part of the lemma. First we perform a formal calculation. Fix an isomorphism $$\mathscr{A}_{x,r}\cong\prod_{\alpha=1}^{n}K'\langle T \rangle$$ where $K'\langle T\rangle$ is the ring of power series with coefficients in $K'$ tending to zero ($T$ a formal variable), and $n$ is a positive integer depending on $r$. Such an isomorphism exists since $\mathbb{B}_{r,K'}$ is just a disjoint union of finitely many affinoid discs. We then have an identification of $\mathscr{D}_{x,r}$ with $\prod_{\alpha=1}^{n}K'\langle[T]\rangle$, where $K'\langle[T]\rangle$ denotes the ring of power series with bounded coefficients in $K'$, and the pairing between $\mathscr{A}_{x,r}$ and $\mathscr{D}_{x,r}$ is given on each component by $\langle \sum a_i T^i, \sum b_j T^j\rangle = \sum a_i b_i$. Now we can compute the action of $\gamma=\begin{pmatrix}
1 & -1\\0 & 1
\end{pmatrix}$ on $\mathscr{D}_{x,r}$. Let $f=(f_\alpha)_{\alpha= 1,...,n}$ be an element of $\mathscr{D}_{x,r}$, with $f_\alpha = \sum b_{j,\alpha} T^j.$ For each $\alpha=1,...,n$ and $i\geq 0$ fix $e_{i,\alpha}$ to be the element of $\mathscr{A}_{x,r}$ which is $T^i$ at the $\alpha$ component, and zero elsewhere. We have \begin{align*}\langle e_{i,\alpha},f\cdot\begin{pmatrix}
1 & -1\\0 & 1
\end{pmatrix}\rangle = \langle T^i, (\sum b_{j,\alpha} T^j)\cdot\begin{pmatrix}
1 & -1\\0 & 1
\end{pmatrix}\rangle &=\langle T^i\cdot\begin{pmatrix}
1 & 1\\0 & 1\end{pmatrix}, \sum b_{j,\alpha} T^j \rangle\\
&= \langle (T+1)^i, \sum b_{j,\alpha} T^j \rangle\\ &= \langle \sum_{k=0}^i \binom{i}{k}T^k,\sum b_{j,\alpha} T^j\rangle \\
&= \sum_{j=0}^i \binom{i}{j}b_{j,\alpha}.\end{align*} 
We can now see that if we have $f=f\cdot \begin{pmatrix}
1 & -1\\0 & 1\end{pmatrix}$ we get $\sum_{j=0}^i \binom{i}{j}b_{j,\alpha}=b_{i,\alpha}$ for all $i$ and $\alpha$, which implies that $f=0$. 

Now we return to the statement in the lemma and suppose $y\in \OCdS$ factors through $Nm$. Let $u_1$ be the element of $U \subset D_f^\times$ with $p$ component equal to $\gamma$ and all other components the identity. For all $g \in D_f^\times$, $y(gu_1)=y(g)\gamma$ since $u_1 \in U$ and $y(gu_1)=y(g)$ since $Nm(u_1)=1$. Hence $y(g)=y(g)\gamma$ and the above calculation shows that $y(g)=0$ for all $g$.

\end{proof}

The following two propositions apply the preceding lemma to give a form of Ihara's lemma for modules of overconvergent automorphic forms and dual overconvergent forms.
\begin{proposition}\label{preihara}
For all but finitely many primes $q\equiv 1$ mod $c$, $T_q-q-1$ annihilates the module $\mathrm{Tor}_1^{\OO(X)}(M/iL^2,\OO(X)/\m)$ for each maximal ideal $\m$ of $\OO(X)$. 
\end{proposition}
\begin{proof}
Fix $q \equiv 1$ mod $c$ with $q \nmid Np\delta l$ and set $H_q:=T_q-q-1$. Let $\m$ be a maximal ideal of $\OO(X)$ and set $K' = \OO(X)/\m$. The maximal ideal $\m$ corresponds to a point $x$ of $X(K')$, with corresponding weight $\kappa_x:\Z_p^\times\rightarrow K'^\times$ the specialisation of $\kappa$ at $\m$.

We have a short exact sequence
$$\xymatrix{
0\ar[r]&L^2\ar[r]^i&M\ar[r]&M/iL^2\ar[r]&0
}.$$
Noting that $L^2$ and $M$ are $\OO(X)$-torsion free, hence flat, and taking derived functors of $-\otimes_{\OO(X)}K'$ gives an exact sequence
$$\xymatrix{
0\ar[r]&\mathrm{Tor}^{\OO(X)}_1(M/iL^2,K')\ar[r]^\delta&L^2\otimes_{\OO(X)}K'\ar@/^2pc/[d]^i\\
0&M/iL^2\otimes_{\OO(X)}K'\ar[l]&M\otimes_{\OO(X)}K'\ar[l]
}.$$

We have a commutative diagram 
$$\xymatrix{
0\ar[r]&L^2\ar[r]^i\ar[d]^{H_q}&M\ar[r]\ar[d]^{H_q}&M/iL^2\ar[r]\ar[d]^{H_q}&0\\
0\ar[r]&L^2\ar[r]^i&M\ar[r]&M/iL^2\ar[r]&0\\
},$$
so by the naturality of the long exact sequence for $\mathrm{Tor}$ the diagram
$$\xymatrix{
\mathrm{Tor}^{\OO(X)}_1(M/iL^2,K')\ar[r]^\delta\ar[d]^{H_q}&L^2\otimes_{\OO(X)}K'\ar[d]^{H_q}\\
\mathrm{Tor}^{\OO(X)}_1(M/iL^2,K')\ar[r]^\delta&L^2\otimes_{\OO(X)}K'
}$$ commutes. To complete the proof it suffices to prove that $H_q$ annihilates the kernel of $$i:L^2\otimes_{\OO(X)}K' \rightarrow M\otimes_{\OO(X)}K'.$$ We proceed by viewing these modules as spaces of automorphic forms with weight $x$ (a single point in weight space). We define two finite dimensional $K'$-vector spaces: \begin{eqnarray*}L_x&:=&\OCS^{\leq d} \\M_x&:=&\OCVS\s.\end{eqnarray*} There are maps $i_x:L_x^2 \rightarrow M_x$ and $i^\dagger_x:M_x \rightarrow L_x$ as defined in section \ref{oldnew} (taking the weight $X$ in that section to be the point $x$), but note that as now the weight is just a point in weight space, Proposition \ref{propinj} does not apply. In particular the map $i_x$ might not be injective.

Recall that $L$ and $M$ are finitely generated $\OO(X)$ modules, whence $L\otimes_{\OO(X)}K'$ and $M\otimes_{\OO(X)}K'$ are finite dimensional $K'$-vector spaces. Since the Newton polygon of the characteristic power series for $U_p$ acting on $\OC$ has the same slope $\le d$ part when specialised to any point of $X$, and the specialisation of the Banach $\OO(X)$-modules $\OC$ and $\OCV$ at $\m$ gives $\OCS$ and $\OCVS$ respectively, we have isomorphisms \\$L\otimes_{\OO(X)}K' \rightarrow L_x$ and $M\otimes_{\OO(X)}K' \rightarrow M_x$ which commute suitably with \\$i:L^2\otimes_{\OO(X)}K' \rightarrow M\otimes_{\OO(X)}K'$ and $i_x:L_x^2 \rightarrow M_x$. These isomorphisms also commute with double coset operators, so to prove the proposition it suffices show that $H_q$ annihilates the kernel of $i_x$. 

Suppose $i_x(y_1,y_2)=0$. Then $y_1=-y_2|\eta_l$, so we have $y_2 \in \OCS$, $y_2|\eta_l \in \OCS$. Therefore $y_2$ and $y_2|\eta_l$ are both invariant under the action of the group $U$, so $y_2$ is invariant under the action of the group generated by $U$ and $\eta_l U \eta_l^{-1}$ in $D_f^\times$. (Note that every element of this group has projection to its $p$th component lying in $\mathbb{M}_\alpha$.) Since by II.1.4, Corollary 1 of \cite{STrees} $$\mathrm{SL}_2(\Q_l)= \langle\mathrm{SL}_2(\Z_l),{\textstyle\begin{pmatrix}
l&0\\0&1
\end{pmatrix}}\mathrm{SL}_2(\Z_l){\textstyle\begin{pmatrix}
l&0\\0&1
\end{pmatrix}}^{-1}\rangle,
$$ and the $l$-factor of $U$ is $\mathrm{GL}_2(\Z_l)$, we have that $y_2$ is invariant under $\mathrm{SL}_2(\Q_l)$, where we embed $\mathrm{SL}_2(\Q_l)$ into $D_f^\times$ in the obvious way. 

Denote by $D^{Nm=1}$ the algebraic subgroup of $D^\times$ whose elements are of reduced norm $1$. We have $D^{Nm=1}(\Q_l) \cong \mathrm{SL}_2(\Q_l)$, since $D$ is split at $l$. The strong approximation theorem applied to $D^{Nm=1}$ implies that $D^{Nm=1}(\Q)\cdot\mathrm{SL}_2(\Q_l)$ is dense in $D^{Nm=1}_f:=D^{Nm=1}(\A_f)$, where $D^{Nm=1}(\Q)$ is embedded diagonally in $D^{Nm=1}_f$. For each $g\in D_f^\times$ we define

$$X^g:=\{h\in D_f^{Nm=1}: y_2(gh)=y_2(g)\}.$$

Since $y_2$ is continuous, $X^g$ is closed, and for $\delta \in D^{Nm=1}(\Q)$, $\gamma \in \mathrm{SL}_2(\Q_l)$ we have $y_2(gg^{-1}\delta\gamma g)=y_2(\delta\gamma g)=y_2(\gamma g)= y_2(g g^{-1}\gamma g) = y_2(g)$, since $g^{-1}\gamma g\in\mathrm{SL}_2(\Q_l)$.  Therefore $X^g$ contains the dense set $g^{-1}D^{Nm=1}(\Q)\mathrm{SL}_2(\Q_l)g$, so $X^g$ is the whole of $D^{Nm=1}_f$. This shows that $y_2$ factors through $Nm$. Now the first part of Lemma \ref{normtrivial} applies.
\end{proof}
\begin{proposition}\label{preiharadual}
The module $\mathrm{Tor}_1^{\OO(X)}(M^*/jL^{*2},\OO(X)/\m)$ is $0$ for all maximal ideals $\m$ of $\OO(X)$
\end{proposition}
\begin{proof}
We again set $K'=\OO(X)/\m$, and let $x \in X(K')$ be the point corresponding to $\m$. Proceeding as at the beginning of the proof of Proposition \ref{preihara} we see that we must show that the map $$j:L^{*2}\otimes_{\OO(X)}K' \rightarrow M^*\otimes_{\OO(X)}K'$$ is injective. 
We define \begin{eqnarray*}L^{*}_x&:=&\OCdS^{\leq d} \\M^*_x&:=&\OCVSd\s.\end{eqnarray*}
As in the previous proposition, it is sufficient to show that the map $j_x:L_x^{*2} \rightarrow M^*_x$ is injective. Now we continue as in the proof of Proposition \ref{preihara}, and finally apply the second part of Lemma \ref{normtrivial}.
\end{proof}

The following consequence of the preceding two propositions will be the most convenient analogue of Ihara's lemma for our applications.
\begin{theorem}\label{lemma:trivialtorsion}
\begin{enumerate} \item There is a positive integer $e$ such that for all but finitely many primes $q\equiv 1 \mod c$, $(T_q-q-1)^{e}$ annihilates $(M/iL^2)^{\mathrm{tors}}$.  Therefore these Hecke operators annihilate the modules $\Lambda_2/\Lambda_3$ and, by consideration of the pairing $P_2$, $\Lambda_0^*/\Lambda_1^*$.
\item The module $(M^*/jL^{*2})^{\mathrm{tors}}$ is equal to $0$. Therefore the modules $\Lambda_2^*/\Lambda_3^*$ and, by consideration of the pairing $P_1$, $\Lambda_0/\Lambda_1$ are also equal to $0$.\end{enumerate}
\end{theorem}
\begin{proof}
We first prove the first part of the theorem. Fix a $q \equiv 1 \mod c$ with $q \nmid Np\delta l$. The module $(M/iL^2)^{\mathrm{tors}}$ is finitely generated (and torsion) over the Dedekind domain $\OO(X)$, so it is isomorphic as an $\OO(X)$-module to $\bigoplus_i \OO(X)/\m_i^{e_i}$ for some finite set of maximal ideals $\m_i$ in $\OO(X)$. We set $e$ to be the maximum of the $e_i$. Set $H_q:=T_q-q-1$ as before. We will show that $H_q^{e}$ annihilates $(M/iL^2)^{\mathrm{tors}}$.

Indeed, suppose that $m \in M$ represents a nonzero torsion class in $M/iL^2$. Thus, there exists a nonzero $\alpha \in \OO(X)$ such that $\alpha m \in iL^2$. We have $$\mathrm{Tor}_1^{\OO(X)}(M/iL^2,\OO(X)/(\alpha))=\{m \in M/iL^2 : \alpha m = 0\}.$$ So we are required to prove that $H_q^e$ annihilates $$\mathrm{Tor}_1^{\OO(X)}(M/iL^2,\OO(X)/(\alpha)).$$
Since $$(M/iL^2)^{\mathrm{tors}}\cong \bigoplus_i \OO(X)/\m_i^{e_i},$$ we can assume that $(\alpha)\supset \prod_i \m_i^{e_i}$, so it is enough to show that $H_q^e$ annihilates $$\bigoplus_i \mathrm{Tor}_1^{\OO(X)}(M/iL^2,\OO(X)/\m_i^{e_i}).$$
 
Taking derived functors of $-\otimes_{\OO(X)} M/iL^2$ of the short exact sequence
$$\xymatrix{
0\ar[r]&\m_i/\m_i^{e_i}\ar[r]&\OO(X)/\m_i^{e_i}\ar[r]&\OO(X)/\m_i\ar[r]&0
}$$
and applying Proposition \ref{preihara} and induction on $e_i$ (note that $\m_i/\m_i^{e_i}$ is isomorphic to $\OO(X)/\m_i^{e_i-1}$), we see that for each $i$, $H_q^{e_i}$ annihilates $\mathrm{Tor}_1^{\OO(X)}(M/iL^2,\OO(X)/\m_i^{e_i})$. Now $e_i \le e$ for all $i$, so $H_q^e$ annihilates all of $\mathrm{Tor}_1^{\OO(X)}(M/iL^2,\OO(X)/(\alpha))$. Since $\alpha$ was arbitrary, $H_q^e$ annihilates $(M/iL^2)^{\mathrm{tors}}$.

The second part of the theorem follows easily from Proposition \ref{preiharadual}.
\end{proof}

\subsection{Supporting Hecke ideals}
We will call a maximal ideal $\mathfrak{M}$ in the Hecke algebra $\mathbb{T}^{(N \delta p)}$ \emph{Eisenstein} (compare \cite{M}) if there is some positive integer $c$ such that for all but finitely many primes $q\equiv 1 \mod c$, $T_q-q-1\in \mathfrak{M}$. The motivation for this definition is the following well-known lemma:

\begin{lemma}Suppose we have a Galois representation $\rho: \mathrm{Gal}(\overline{\Q}/\Q)\rightarrow \mathrm{GL}_2(\overline{\Q}_p)$, continuous and unramified at all but finitely many primes, and a positive integer $c$ such that for all but finitely many primes $q \equiv 1 \mod c$, $\rho$ is unramified at $q$ and the trace of Frobenius at $q$, $\mathrm{Tr}(\mathrm{Frob}_q)=q+1$. Then $\rho$ is reducible.
\end{lemma}
\begin{proof} The formula for the traces implies that $\rho$ restricted to the absolute Galois group of the cyclotomic field $\Q(\zeta_c)$ has semisimplification isomorphic to $\mathbf{1}\oplus \chi$, where $\mathbf{1}$ is the trivial one dimensional representation, and $\chi$ is the $p$-adic cyclotomic character. Denoting $\mathrm{Gal}(\overline{\Q}/\Q)$ by $G$ and $\mathrm{Gal}(\overline{\Q}/\Q(\zeta_c))$ by $H$  and applying Frobenius reciprocity we conclude that $$\mathrm{Hom}(\mathrm{Ind}_H^G(\mathbf{1}\oplus \chi),\rho)$$ is non-zero. $\mathrm{Ind}_H^G(\mathbf{1}\oplus \chi)$ is just a direct sum of one dimensional representations, so $\rho$ is reducible.
\end{proof}
We recall that for an arbitrary (commutative) ring $R$ the \emph{support} of an $R$-module $A$ is the set of prime ideals $\mathfrak{p}\lhd R$ such that the localisation $A_\mathfrak{p}$ is non-zero. If $A$ is finitely generated as an $R$-module then the support of $A$ is equal to the set of prime ideals in $R$ containing the annihilator of $A$. We write $\mathbb{T}_L$ for the image of $\mathbb{T}^{(N \delta p)}$ in $\mathrm{End}_{\OO(X)}(L)$ and similarly $\mathbb{T}_M$ for the image of $\mathbb{T}^{(N \delta p l)}$ in $\mathrm{End}_{\OO(X)}(M)$. Analogously we define $\mathbb{T}_{L^*}$ and $\mathbb{T}_{M^*}$ to be the image of the Hecke algebra in the endomorphism rings of the relevant dual modules. Note that there are natural maps $\mathbb{T}_M\rightarrow \mathbb{T}_L$ and $\mathbb{T}_{M^*}\rightarrow \mathbb{T}_{L^*}$. If $I$ is an ideal of $\mathbb{T}^{(N\delta p)}$ we write $I_L$ for the image of $I$ in $\mathbb{T}_L$ and $I'_M$ for the image of $I\cap \mathbb{T}^{(N\delta pl)}$ in $\mathbb{T}_M$

 We can now state and prove the main theorem of this section. 
\begin{theorem}\label{supporttheorem}Suppose $\mathfrak{M}$ is a non-Eisenstein maximal ideal of $\mathbb{T}^{(N\delta p)}$ containing $T_l^2-(l+1)^2 S_l$. Further suppose that $\mathfrak{M}_L$ is in the support of the $\mathbb{T}_L$-module $L$. Then $\mathfrak{M}'_M$ is in the support of the $\mathbb{T}_M$-module $\mathrm{ker}(i^\dagger)\subset M$.
\end{theorem}
\begin{proof}
We write $\mathfrak{M}_{L^*}$ for the image of $\mathfrak{M}$ in $\mathbb{T}_{L^*}$ and $\mathfrak{M}'_{M^*}$ for the image of $\mathfrak{M}\cap \mathbb{T}^{(N\delta pl)}$ in $\mathbb{T}_{M^*}$. Since $\mathfrak{M}_L$ is in the support of $L$, and the perfect pairing between $L$ and $L^*$ is equivariant with respect to all of $\mathbb{T}^{(N\delta p)}$ (including $T_l$), we know that $\mathfrak{M}_{L^*}$ is in the support of $L^*$. Consider the module $$Q:=\Lambda_0^*/\Lambda_3^*=L^{*2}\left/
L^{*2}\begin{pmatrix}l+1 & T_l \\ S_l^{-1}T_l & l+1 \end{pmatrix},\right.$$ Since $L^{*2}_{\mathfrak{M}_{L^*}}\neq 0$ and $\mathrm{det}\begin{pmatrix}l+1 & T_l \\ S_l^{-1}T_l & l+1 \end{pmatrix}\in \mathfrak{M}$ we know that $Q_{\mathfrak{M}_{L^*}}\neq 0$, i.e. $\mathfrak{M}_{L^*}$ is in the support of $Q$. We can view $Q$ as a $\mathbb{T}_{M^*}$-module, with $\M'_{M^*}$ in its support. 

Theorem \ref{lemma:trivialtorsion} implies that if $\M'_{M^*}$ is in the support of $\Lambda_0^*/\Lambda_1^*$ or $\Lambda_2^*/\Lambda_3^*$ then it is Eisenstein, so it must be in the support of $\Lambda_1^*/\Lambda_2^*$. This quotient is a homomorphic image of $M^*/(M^*\cap j(L_F^{*2}))$, so $\M'_{M^*}$ is in the support of $M^*/(M^*\cap j(L_F^{*2}))$. Finally we can apply pairing $P_3$ (which is equivariant with respect to $\mathbb{T}^{(N\delta pl)}$) to conclude that $\M'_M$ is in the support of $\mathrm{ker}(i^\dagger)$. 
\end{proof}

\section{Applications}\label{apps}
In this section we explain some applications of the preceding results, including the proof of some cases of the conjecture mentioned in the introduction.

\subsection{Geometric level raising for $p$-adic modular forms}
We firstly describe the application of Theorem \ref{supporttheorem} to the conjecture of our Introduction. Chenevier \cite{MR2111512} extended the classical Jacquet-Langlands correspondence to a rigid analytic embedding from the eigencurve for a definite quaternion algebra to some part of the $\mathrm{GL}_2$ eigencurve. We may use this to translate the results of the previous section to the $\mathrm{GL}_2$ eigencurve. 

We state the case of Theorem $3$ in \cite{MR2111512} that we will use. We have primes $p,l$ and a coprime integer $N \ge 1$. Fix another prime $q$ and a character $\varepsilon$ of $\Z/Np\Z$. Let $\mathscr{E}$ be the tame level $\Gamma_1(N)\cap\Gamma_0(q)$ and character $\varepsilon$ reduced cuspidal eigencurve. Let $D/\Q$ be a quaternion algebra ramified at the infinite place and at $q$, and let $\mathscr{E}^D$ be the corresponding reduced eigencurve of tame level $U_1(N)$ and character $\varepsilon$. 

\begin{theorem}
There is a closed rigid analytic immersion $$JL_p: \mathscr{E}^D \hookrightarrow \mathscr{E}$$ whose image is the Zariski closure of the classical points of $\mathscr{E}$ that are new at $q$. This map is defined over weight space and is Hecke equivariant.
\end{theorem}

We get the same result if we change the level of $\mathscr{E}$ to $\Gamma_1(N)\cap\Gamma_0(ql)$ (call this eigencurve $\mathscr{E}'$) and change the level of $\mathscr{E}^D$ to $U_1(N)\cap U_0(l)$ (call this eigencurve $\mathscr{E}^{D'}$), where we construct these eigencurves using the Hecke operators at $l$ in addition to the usual Hecke operators away from the level. This allows us to relate $\mathscr{E}'$ and the two-covering $\mathscr{E}^{\mathrm{old}}$ of $\mathscr{E}$ corresponding to taking roots of $l$th Hecke polynomial.

\begin{lemma}\label{OLD}
There is a closed embedding $\mathscr{E}^{\mathrm{old}}\hookrightarrow \mathscr{E}'$, with image the Zariski closure of the classical $l$-old points in $\mathscr{E}'$.
\end{lemma}
\begin{proof}
Given $X$ an affinoid subdomain of $\mathscr{W}$, let $M_X$ and $M_X'$ denote the Banach $\OO(X)$-modules of families (weight varying over $X$) of overconvergent modular forms of tame levels $\Gamma_1(N)\cap\Gamma_0(q)$ and $\Gamma_1(N)\cap\Gamma_0(ql)$ respectively, as defined in section 7 of \cite{Bu2}. The two degeneracy maps from level $\Gamma_1(N)\cap\Gamma_0(ql)$ to level $\Gamma_1(N)\cap\Gamma_0(q)$ give a natural embedding $M_X^2 \rightarrow M_X'$. Denote the image of this map by $M_X^\mathrm{old}$ - it is stable under all Hecke operators (including at $l$). The lemma follows from the observation that applying the eigenvariety machine of \cite{Bu2} to the Banach modules $M_X^\mathrm{old}$ (with $X$ varying) gives the space $\mathscr{E}^{\mathrm{old}}$. 
\end{proof}

Let $\mathscr{Z}\subset \mathscr{E}'$ be the Zariski closure of the classical points in $\mathscr{E}'$ corresponding to forms new at $l$. Proposition 4.7 of \cite{MR2111512} shows that $\mathscr{Z}$ can be identified with the points of $\mathscr{E}'$ lying in a one dimensional family of \emph{$l$-new} points, where $l$-new means they come from overconvergent modular forms in the kernel of the map analogous to $i^\dagger$ in the $GL_2$ setting. The following theorem corresponds to the conjecture in the introduction for points of $\mathscr{E}'$ in the image of $JL_p$.

\begin{theorem}\label{raise}
Suppose we have a point $\phi \in \mathscr{E}$ lying in the image of $JL_p$, with $T_l^2(\phi)-(l+1)^2S_l(\phi)=0$. Let the roots of the $l$th Hecke polynomial corresponding to $\phi$ be $\alpha$ and $l \alpha$ where $\alpha \in \C_p$. Then the point over $\phi$ of $\mathscr{E}^\mathrm{old}$ corresponding to $\alpha$ lies in $\mathscr{Z}$.
\end{theorem}
\begin{proof}

We pick $d$, $r$ and $\alpha$ such that the automorphic form corresponding to the preimage of $\phi$ under $JL_p$ is $r$-overconvergent of slope $\le d$ and level $\UU$. Now fix a closed ball in $\mathscr{W}$, containing the weight of $\phi$, which is small enough (note that `small enough' depends on $d$, $r$ and $\alpha$) to apply the local eigenvariety construction described in section 6.2 of \cite{Chenun}. Denote this ball by $X$. As usual the $rp^{-\alpha}$-analytic character $\Z_p^\times \rightarrow \OO(X)$ induced by the embedding $X \hookrightarrow \mathscr{W}$ is denoted by $\kappa$. 
 
The system of Hecke eigenvalues given by $\phi$ corresponds to a maximal ideal $\mathfrak{M}$ in the $\OO(X)$ Hecke algebra $\mathbb{T}^{(Npq)}$. We know that $T_l^2-(l+1)^2S_l \in \mathfrak{M}$. If we set $L=\OCN\s$ as before, and use the notation of the previous section, then $\mathfrak{M}_L$ is in the support of $L$. At this stage Theorem \ref{supporttheorem} applies to the ideal $\mathfrak{M}$, which is not Eisenstein since the Galois representation attached to $\phi$ is irreducible (recall that we are working on the cuspidal part of the eigencurve). Therefore we know that $\mathfrak{M}'_M$ is in the support of $\mathrm{ker}(i^\dagger)\subset M$. We can then take a height one prime ideal $\mathfrak{p}\subset \mathfrak{M}'_M$ in the support of $\mathrm{ker}(i^\dagger)\subset M$, and then this corresponds (by proposition 6.2.4 of \cite{Chenun}) to a $p$-adic family of automorphic forms, new at $l$, passing through a point $\phi'$ with system of Hecke eigenvalues the same as those for $\phi$ away from $l$. Now applying the map $JL_p$ we see that one of the points over $\phi$ must lie in $\mathscr{Z}$. A calculation using the fact that $\phi'$ comes from an eigenform in the kernel of $i^\dagger$ shows that the point corresponds to the root $\alpha$.
\end{proof}
To translate this theorem into the language of the introduction, note that $\mathscr{E}^\mathrm{old}$ corresponds to the generically unramified principal series components of $\mathscr{E}'$, whilst $\mathscr{Z}$ corresponds to the generically special or supercuspidal components. We may identify the point over $\phi$ lying in $\mathscr{Z}$ as the one whose attached $\mathrm{GL}_2(\Q_l)$-representation is special. 

\subsection{Eigenvarieties of newforms}\label{neweigen}
We return to the situation of a definite quaternion algebra $D$ over $\Q$ with arbitrary discriminant $\delta$ prime to $p$. Denote the levels $\UU$ by $U$ and $\VV$ by $V$, as before. We denote by $\mathscr{E}^D$ the tame level $U_1(N)\cap U_0(l)$ reduced eigencurve for $D$. Suppose $\phi$ is a point of $\mathscr{E}^D$, with weight $x$. We say that $\phi$ is \emph{$l$-new} if it corresponds to a Hecke eigenform in the kernel of the map $i^\dagger_x: \OCVS \rightarrow \OCS$, where this is defined as in section \ref{oldnew} by $i^\dagger_x (f):=(f|[V 1 U],f|[V \eta_l^{-1} U])$. Denote by $\mathscr{Z}$ the Zariski closure of the points in $\mathscr{E}^D$ arising from \emph{classical} $l$-new forms. We have the following proposition, due to Chenevier.

\begin{proposition}
\begin{enumerate}\item The set of $x$ in $\mathscr{E}^D$ that are $l$-new is the set of points of a closed reduced analytic subspace $\mathscr{E}^D_{\mathrm{new}} \subset \mathscr{E}^D$. 
\item $\mathscr{Z}$ is a closed subspace of $\mathscr{E}^D_{\mathrm{new}}$, and its complement is the union of irreducible components of dimension $0$ in $\mathscr{E}^D_{\mathrm{new}}$. 
\item A point of $\mathscr{E}^D_{\mathrm{new}}$ lies in $\mathscr{Z}$ if and only if it lies in a one-dimensional family of points in $\mathscr{E}^D_{\mathrm{new}}$.\end{enumerate}\end{proposition}
\begin{proof}
Exactly as for Proposition 4.7 of \cite{MR2111512}.
\end{proof}

We now apply the results of section \ref{mainbody} to show that a point of $\mathscr{E}^D_{\mathrm{new}}\backslash \mathscr{Z}$ lies in a one-dimensional family of points in $\mathscr{E}^D_{\mathrm{new}}$, so by contradiction we can conclude that $\mathscr{E}^D_{\mathrm{new}}$ is equal to $\mathscr{Z}$.

\begin{theorem}\label{thm:neweigen}
$\mathscr{E}^D_{\mathrm{new}}$ is equal to $\mathscr{Z}$. In particular, $\mathscr{E}^D_{\mathrm{new}}$ is equidimensional of dimension $1$.
\end{theorem}
\begin{proof}
Let $\phi$ be a point of $\mathscr{E}^D_{\mathrm{new}}\backslash \mathscr{Z}$ with weight $x$. The proof will proceed by showing that $\phi$ is also $l$-old, then raising the level at $\phi$, as in the previous theorem, to show that it lies in a family of $l$-new points.

We pick $d$, $r$ and $\alpha$ such that $\phi$ comes from a $r$-overconvergent automorphic form of slope $d$ and level $U=U_1(Np^\alpha)$. Now fix a closed ball in $\mathscr{W}$, containing the weight of $\phi$, which is small enough (note that `small enough' depends on $d$, $r$ and $\alpha$) to apply the local eigenvariety construction described in section 6.2 of \cite{Chenun}. Denote this ball by $X$. The point $x$ in $X$ corresponds to a maximal ideal $\m$ of $\OO(X)$. If we set $M=\OCV\s$ as before, then $\phi$ lies in a family corresponding to a Hecke eigenvector in $M$.

As in the previous subsection, we have a closed embedding $\mathscr{E}^D_{\mathrm{old}}\hookrightarrow \mathscr{E}^D$, where $\mathscr{E}^D_{\mathrm{old}}$ is a two-covering of the tame level $U$ reduced eigencurve for $D$, and the image of this embedding is the Zariski closure of the classical $l$-old points in $\mathscr{E}^D$, which also equals the space of $l$-old points in $\mathscr{E}^D$ (as is clear from applying the proof of Lemma \ref{OLD} to modules of overconvergent automorphic forms for $D$). Since the space $\mathscr{E}^D$ is equidimensional of dimension $1$, but $\phi$ does not lie in a family of points in $\mathscr{E}^D_{\mathrm{new}}$, $\phi$ must lie in a family of points in $\mathscr{E}^D_{\mathrm{old}}$. So $\phi$ is $l$-old and $l$-new.

We now `raise the level' at $\phi$. As in the proof of Proposition \ref{preihara} we specialise the $\OO(X)$ modules $L=\OC\s$ and $M$ at the maximal ideal $\m$ to give vector spaces $L_x$ and $M_x$. Since $\phi$ is $l$-new and $l$-old, it arises from an eigenform $g$ in $\mathrm{im}(i_x)\cap \ker (i^\dagger_x)$. Now a calculation using the explicit matrix for the map $i^\dagger_x i_x$ shows that $g$ is of the form $i_x(\alpha f,-f)$, where $f$ is an eigenform in $L_x$ with $(T_l^2-(l+1)^2S_l)f=0$, and the roots of the $l$th Hecke polynomial for $f$ are $\alpha$ and $l\alpha$ (note that with our normalisations $g$ is the $l$-stabilisation of $f$ corresponding to the root $\alpha$). Now applying the proof of Theorem \ref{raise} we see that $\phi$ lies in $\mathscr{Z}$, so we have a contradiction and therefore must have $\mathscr{E}^D_{\mathrm{new}}=\mathscr{Z}$.
\end{proof}

\subsubsection{Modules of newforms and the eigenvariety machine}Define the Banach $\OO(X)$-module $M_{X,r}^{\mathrm{new}}$ for varying affinoid subdomains $X \subset \mathscr{W}$ (with corresponding character $\kappa$) to be the kernel of the map $$i^\dagger: \OCNl\rightarrow \OCN\times\OCN.$$ One might wish to construct the rigid analytic space $\mathscr{E}^D_{\mathrm{new}}$ from these modules using Buzzard's eigenvariety machine \cite{Bu2}. The key issue is to show that the modules $M_{X,r}^{\mathrm{new}}$ behave well under base change between affinoid subdomains. The author is not sure whether this should be true or not - the results in this paper can be viewed as showing that these modules behave well under base change from a sufficiently small affinoid to a point and it is not clear that one can conclude something about base change between open affinoids from this.

\section{Acknowledgments}
This paper is the result of research carried out whilst studying for a PhD under the supervision of Kevin Buzzard, to whom I am grateful for suggesting the problem and providing such excellent guidance. I would also like to thank Owen Jones, David Loeffler and Kevin McGerty for their helpful comments, and the referee for several useful remarks and corrections.

\bibliography{dissbib}
\end{document}